\documentclass[12pt]{amsart}
\usepackage{amsmath,amssymb,amsbsy,amsfonts,latexsym,amsopn,amstext,cite,
                                               amsxtra,euscript,amscd,bm,mathabx}
\usepackage{url}
\usepackage[colorlinks,linkcolor=blue,anchorcolor=blue,citecolor=blue,backref=page]{hyperref}
\usepackage{color}
\usepackage{graphics,epsfig}
\usepackage{graphicx}
\usepackage{float} 
\usepackage[english]{babel} 
\usepackage{mathtools}
\usepackage{todonotes}
\usepackage{url}
\usepackage[colorlinks,linkcolor=blue,anchorcolor=blue,citecolor=blue,backref=page]{hyperref}

\usepackage[norefs,nocites]{refcheck}

\hypersetup{breaklinks=true}

\usepackage[norefs,nocites]{refcheck}
\usepackage[english]{babel}
\begin{document}

\newtheorem{thm}{Theorem}
\newtheorem{lem}[thm]{Lemma}
\newtheorem{claim}[thm]{Claim}
\newtheorem{cor}[thm]{Corollary}
\newtheorem{prop}[thm]{Proposition} 
\newtheorem{definition}[thm]{Definition}
\newtheorem{question}[thm]{Open Question}
\newtheorem{qn}[thm]{Question}
\newtheorem{conj}[thm]{Conjecture}
\newtheorem{prob}{Problem}

\theoremstyle{remark}
\newtheorem{rem}[thm]{Remark}

\newcommand{\GL}{\operatorname{GL}}
\newcommand{\SL}{\operatorname{SL}}
\newcommand{\lcm}{\operatorname{lcm}}
\newcommand{\ord}{\operatorname{ord}}
\newcommand{\Op}{\operatorname{Op}}
\newcommand{\Tr}{\operatorname{Tr}}
\newcommand{\Nm}{\operatorname{Nm}}

\numberwithin{equation}{section}
\numberwithin{thm}{section}
\numberwithin{table}{section}

\def\vol {{\mathrm{vol\,}}}
\def\squareforqed{\hbox{\rlap{$\sqcap$}$\sqcup$}}
\def\qed{\ifmmode\squareforqed\else{\unskip\nobreak\hfil
\penalty50\hskip1em\null\nobreak\hfil\squareforqed
\parfillskip=0pt\finalhyphendemerits=0\endgraf}\fi}

\def \balpha{\bm{\alpha}}
\def \bbeta{\bm{\beta}}
\def \bgamma{\bm{\gamma}}
\def \blambda{\bm{\lambda}}
\def \bchi{\bm{\chi}}
\def \bphi{\bm{\varphi}}
\def \bpsi{\bm{\psi}}
\def \bomega{\bm{\omega}}
\def \btheta{\bm{\vartheta}}

\def\eps{\varepsilon}

\newcommand{\bfxi}{{\boldsymbol{\xi}}}
\newcommand{\bfrho}{{\boldsymbol{\rho}}}

\def\Kab{\sfK_\psi(a,b)}
\def\Kuv{\sfK_\psi(u,v)}
\def\SaUV{\cS_\psi(\balpha;\cU,\cV)}
\def\SaAV{\cS_\psi(\balpha;\cA,\cV)}

\def\SUV{\cS_\psi(\cU,\cV)}
\def\SAB{\cS_\psi(\cA,\cB)}

\def\Kmnp{\sfK_p(m,n)}

\def\KKap{\cH_p(a)}
\def\KKaq{\cH_q(a)}
\def\KKmnp{\cH_p(m,n)}
\def\KKmnq{\cH_q(m,n)}

\def\Klmnp{\sfK_p(\ell, m,n)}
\def\Klmnq{\sfK_q(\ell, m,n)}

\def \SALMNq {\cS_q(\balpha;\cL,\cI,\cJ)}
\def \SALMNp {\cS_p(\balpha;\cL,\cI,\cJ)}

\def \SACXMQX {\fS(\balpha,\bzeta, \bxi; M,Q,X)}

\def\SAMJp{\cS_p(\balpha;\cM,\cJ)}
\def\SAMJq{\cS_q(\balpha;\cM,\cJ)}
\def\SAqMJq{\cS_q(\balpha_q;\cM,\cJ)}
\def\SAJq{\cS_q(\balpha;\cJ)}
\def\SAqJq{\cS_q(\balpha_q;\cJ)}
\def\SAIJp{\cS_p(\balpha;\cI,\cJ)}
\def\SAIJq{\cS_q(\balpha;\cI,\cJ)}

\def\RIJp{\cR_p(\cI,\cJ)}
\def\RIJq{\cR_q(\cI,\cJ)}

\def\TWXJp{\cT_p(\bomega;\cX,\cJ)}
\def\TWXJq{\cT_q(\bomega;\cX,\cJ)}
\def\TWpXJp{\cT_p(\bomega_p;\cX,\cJ)}
\def\TWqXJq{\cT_q(\bomega_q;\cX,\cJ)}
\def\TWJq{\cT_q(\bomega;\cJ)}
\def\TWqJq{\cT_q(\bomega_q;\cJ)}

 \def \xbar{\overline x}
  \def \ybar{\overline y}

\def\cA{{\mathcal A}}
\def\cB{{\mathcal B}}
\def\cC{{\mathcal C}}
\def\cD{{\mathcal D}}
\def\cE{{\mathcal E}}
\def\cF{{\mathcal F}}
\def\cG{{\mathcal G}}
\def\cH{{\mathcal H}}
\def\cI{{\mathcal I}}
\def\cJ{{\mathcal J}}
\def\cK{{\mathcal K}}
\def\cL{{\mathcal L}}
\def\cM{{\mathcal M}}
\def\cN{{\mathcal N}}
\def\cO{{\mathcal O}}
\def\cP{{\mathcal P}}
\def\cQ{{\mathcal Q}}
\def\cR{{\mathcal R}}
\def\cS{{\mathcal S}}
\def\cT{{\mathcal T}}
\def\cU{{\mathcal U}}
\def\cV{{\mathcal V}}
\def\cW{{\mathcal W}}
\def\cX{{\mathcal X}}
\def\cY{{\mathcal Y}}
\def\cZ{{\mathcal Z}}
\def\Ker{{\mathrm{Ker}}}
\def\g{{\mathrm{gcd}}}

\def\NmQR{N(m;Q,R)}
\def\VmQR{\cV(m;Q,R)}

\def\Xm{\cX_m}

\def \A {{\mathbb A}}
\def \B {{\mathbb A}}
\def \C {{\mathbb C}}
\def \N {{\mathbb N}}
\def \F {{\mathbb F}}
\def \G {{\mathbb G}}
\def \L {{\mathbb L}}
\def \K {{\mathbb K}}
\def \PP {{\mathbb P}}
\def \Q {{\mathbb Q}}
\def \R {{\mathbb R}}
\def \Z {{\mathbb Z}}
\def \fS{\mathfrak S}

\def\e{{\mathbf{\,e}}}
\def\ep{{\mathbf{\,e}}_p}
\def\eq{{\mathbf{\,e}}_q}
\def\er{{\mathbf{\,e}}_R}
\def\esr{{\mathbf{\,e}}_r}
\def\\{\cr}
\def\({\left(}
\def\){\right)}
\def\fl#1{\left\lfloor#1\right\rfloor}
\def\rf#1{\left\lceil#1\right\rceil}

\def\Tr{{\mathrm{Tr}}}
\def\Nm{{\mathrm{Nm}}}
\def\Im{{\mathrm{Im}}}

\def \oF {\overline \F}

\newcommand{\pfrac}[2]{{\left(\frac{#1}{#2}\right)}}

\def \Prob{{\mathrm {}}}
\def\e{\mathbf{e}}
\def\ep{{\mathbf{\,e}}_p}
\def\epp{{\mathbf{\,e}}_{p^2}}
\def\em{{\mathbf{\,e}}_m}

\def\Res{\mathrm{Res}}
\def\Orb{\mathrm{Orb}}

\def\vec#1{\mathbf{#1}}
\def \va{\vec{a}}
\def \vb{\vec{b}}
\def \vm{\vec{m}}
\def \vu{\vec{u}}
\def \vv{\vec{v}}
\def \vx{\vec{x}}
\def \vy{\vec{y}}
\def \vz{\vec{z}}
\def\flp#1{{\left\langle#1\right\rangle}_p}
\def\T {\mathsf {T}}

\def\sfG {\mathsf {G}}
\def\sfK {\mathsf {K}}

\def\mand{\qquad\mbox{and}\qquad}

\title[Character sums of Division polynomials]
{Character sums of Division polynomials twisted by multiplicative functions}

\author{Subham Bhakta}
\address{School of Mathematics and Statistics, University of New South Wales, Sydney, NSW 2052, Australia.} 
\email{subham.bhakta@unsw.edu.au}


\begin{abstract}
Let $E$ be an elliptic curve over the finite field $\mathbb{F}_p$, and $P \in E(\mathbb{F}_p)$ be an $\F_p$-rational point. We obtain nontrivial estimates for multiplicative character sums associated with the division polynomials $\psi_n(P)$ twisted by several multiplicative functions, where $\psi_n(P)$ denotes the $n$-th division polynomial evaluated at $P$.
\end{abstract}



\maketitle
\tableofcontents
\section{Introduction}

\subsection{Set-up}

Let $E$ be an elliptic curve over $ \F_p $ of characteristic $ p > 3 $ given by the Weierstrass equation as follows, with coefficients in $\F_p$.
\[
y^2=x^3+ax+b.
\]
For any integer $n\geq0$, define the $n$th division polynomial $\psi_n\in\F_p[x,y]$ as follows.
\begin{align*}
\psi_0 &= 0, \quad \psi_1 = 1, \quad \psi_2 = 2y, \\
\psi_3 &= 3x^4 + 6a x^2 + 12b x - a^2, \\
\psi_4 &= 4y(x^6 + 5a x^4 + 20b x^3 - 5a^2 x^2 - 4abx - 8b^2 - a^3),
\end{align*}
for more details, see \cite[Exercise 3.7]{Silv2}, combined with the discussion in \cite[Chapter III.1]{Silv2}.

Denote by $ E(\F_p) $ the group of points on $ E $ defined over $ \F_p $. We can interpret each $\psi_n$ as a rational function on $E(\F_p)$. The multiplication of $P \in E$ by $n$ is given as a rational map by 
\begin{align*}
[n](P) &=
\Bigl(
\frac{x(P)\psi_n^2(P)  - \psi_{n-1}(P)\psi_{n+1}(P)}{\psi_n(P)^2}, \\
& \qquad \qquad \frac{\psi_{n-1}(P)^2\psi_{n+2}(P) - \psi_{n-2}(P)\psi_{n+1}(P)^2}{4y(P)\psi_n(P)^3}
\Bigr).
\end{align*}
It is well known that
\begin{align*}
\psi_{m+n}(P)&\psi_{m-n}(P)\psi_r^2(P)\\
& =\psi_{m+r}(P)\psi_{m-r}(P)\psi_n^2(P)-\psi_{n+r}(P)\psi_{n-r}(P)\psi_m(P)^2,
\end{align*}
for any integers $m,n,r$, see, for example,~\cite[Exercise~3.7(g)]{Silv2}.
 It turns out that $\psi_n(P)$ is periodic, and the period could be as large as $(p-1) \ord P$, 
 where $\ord P$ is the order of $P$ in the group $E(\F_p)$, see~\cite[Corollary~9]{Silv1}.   
 
Let $\chi$ be a multiplicative character of $\F_p^{*}$ of a small order. We consider the sequence $\chi(\psi_n(P))$, where of course we set $\chi(0)=0$. Then, $\chi(\psi_n(P))$ is periodic with a much smaller period (see Section~\ref{sec:elliptic}), compared to the period of $\psi_n(P)$. For points, $P\in E(\F_p)$ of large orders, Shparlinski and Stange~\cite{ShSt} have obtained non-trivial estimates for the character sums of the form 
\begin{equation*}
\sum_{1\le n\le N} \chi(\psi_n(P)),
\end{equation*}
for quadratic characters $\chi$ of $\F_p$, and obtained a nontrivial bound provided $p^{1/2+\varepsilon}\le N\le \ord P$, for some fixed $\varepsilon > 0$. 

Studying the character values $\chi(\psi_n(P))$ has implications for understanding the distribution of rational points on elliptic curves over $\mathbb{Q}$. More specifically, let $E/\mathbb{Q}$ be an elliptic curve given by an integral Weierstrass equation, and let $P \in E(\mathbb{Q})$ be a non-torsion integral point with everywhere good reduction (see \cite[Section~5.1]{BLMN} for an example). We can write the points $nP \in E(\mathbb{Q})$ in lowest terms as
\[
nP = \left( \frac{a_{nP}}{d_{nP}^2}, \frac{b_{nP}}{d_{nP}^3} \right),
\]
where $d_{nP} \in \mathbb{N}$ and $\gcd(a_{nP} b_{nP}, d_{nP}) = 1$. Then \cite[Theorem~A]{Ayad} or \cite[Proposition~3.4]{Verzobio} implies that $d_{nP} = |\psi_n(P)|$, where $\psi_n$ is the $n$th division polynomial associated to $E/\Q$. Consequently, at least for even Dirichlet characters $\chi$ of prime modulus $p$, we study $\chi(d_{nP})$, i.e., the distribution of $d_{nP} \bmod p$. See also some related discussions in \cite[Section~6]{BLMN}.

In this article, we study the correlation between the character values $\chi(\psi_n(P))$ and various multiplicative functions. More precisely, for a given point $P\in E(\F_p)$, we study the twisted sums of the form
\[
    S_{f,\chi,P}(N) = \sum_{1 \leq n\leq N} f(n)\chi(\psi_n(P)),
\]  
where $f$ is a suitable multiplicative function.

\subsection{Main results}\label{sec:results}
\subsubsection{Preliminaries}
Twists of trace functions with suitable multiplicative functions have been studied by Korolev and Shparlinski~\cite{KorShp}. A key ingredient of their study involves a certain correlation property (see \cite[Corollary~4.2]{KorShp}) of trace functions. Building upon the work of Shparlinski and Stange \cite{ShSt}, we first show in Lemma~\ref{lem:corel} that our function $\chi(\psi_n(P))$ also enjoys similar properties. 

Throughout the whole article, we set $\chi$ to always denote a  multiplicative character on $\F_p^{*}$ of order $d$,
and 
\begin{equation}\label{eqn:R}
R=d  \ord P. 
\end{equation}

For $ n \geq 1 $ and $ \nu \geq 1 $, let $ \tau_{\nu}(n) $ denote the $ \nu $-fold divisor function, defined as  
\[
\tau_{\nu}(n) = \sum_{n_1n_2 \cdots n_{\nu} = n} 1,
\]
where the sum runs over ordered $ \nu $-tuples $ (n_1, n_2, \ldots, n_{\nu}) $ of positive integers such that $ n_1n_2 \cdots n_{\nu} = n $. In other words, $ \tau_{\nu}(n) $ is the coefficient of $ n^{-s} $ in the Dirichlet series  
\[
\zeta(s)^v = \sum_{n=1}^\infty \tau_{\nu}(n) n^{-s}.
\]
\subsubsection{A general bound}
For any fixed $C\geq 1$, let us recall the well-known estimate from \cite[Equation (1.80)]{IwKow}
\begin{equation}\label{eq:sumtau}
\sum_{1\leq n\leq N}\tau_{\nu}(n)^{C}\ll N(\log N)^{\nu^{C}-1}.
\end{equation}

We say that a multiplicative function $f : \N \to \C$ is a $\tau_{\nu}$-bounded multiplicative function if $|f(n)|\le \tau_{\nu}(n)$ for all $n \in \N$. We refer the reader to \cite[Sections 1.2 and 1.3]{Man} for examples of various $\tau_{\nu}$-bounded multiplicative functions. Analogous to \cite[Theorem 2.2]{KorShp}, we have the following result. 

\begin{thm}\label{thm:divbdd}
Let $0<\varepsilon\leq 1/2$ be a fixed real number, and assume that
\begin{equation}\label{eqn:Rcondition}
R\ge p^{1/2+\varepsilon}.
\end{equation}
Let $\nu\geq 1$ be an integer. Then, for any $\tau_{\nu}$-bounded multiplicative function $f(n)$, we have
\begin{equation*}
S_{f,\chi,P}(N) \ll \varepsilon^{-\nu} N \frac{(\log \log R)^{\nu}}{\log R},
\end{equation*}
provided that 
\begin{equation}\label{eqn:Nfor1bdd}
    \quad p^{1/12}R^{5/6+\varepsilon}\leq N\leq R.
\end{equation}
\end{thm}

Note that \eqref{eqn:Nfor1bdd} enforces \eqref{eqn:Rcondition}. Moreover, Remark~\ref{eq:quadchars} and the proof of Theorem~\ref{thm:divbdd} together show that, for any quadratic character~$\chi$, the same bound for~$S_{f,\chi,P}(N)$ holds, but rather over the extended range~$p^{1/2+\varepsilon} \leq N \leq R$.

\subsubsection{A power saving for certain $1$-bounded $f$}
When $f$ is a Dirichlet character, we get a power savings compared to the trivial bound in certain ranges of $R$. More precisely, we prove the following estimate.
\begin{thm}\label{thm:dirichlet}
Let $R$ be as in \eqref{eqn:R}, and assume that
\begin{equation}\label{eq:Rweaker}
R\geq p^{1/2}\exp\(2.1\log p/\log \log p\).
\end{equation}
Then, for any fixed Dirichlet character $\psi$, we have
    \begin{equation*}
S_{\psi,\chi,P}(N) \ll p^{1/12}R^{5/6}(\log R)(\log \log R)^{1/3}.
\end{equation*}
\end{thm}

One can easily note from the argument in Section~\ref{sec:concludingThm1} that the weaker condition \eqref{eq:Rweaker} does not work in Theorem~\ref{thm:divbdd}.

Furthermore, we study the sums twisted by $f=\mu^2$, where as usual, $\mu$ denotes the M{\"o}bius functions, and obtain a power savings in certain ranges. For this, we study certain Type-I sums associated with $\chi(\psi_n(P))$, as done for the proof of \cite[Theorem 1.1]{ShaShpWij}, and prove the following estimate.

\begin{thm}\label{thm:sqfree}
   Let $R$ and $\varepsilon$ be as in \eqref{eqn:Rcondition}. Then for any integer $1\leq N\leq R$, we have
\begin{align*}
S_{\mu^2,\chi,P}(N)&\ll N^{1/2}p^{1/24}R^{5/12}(\log R)^{3/2}(\log \log R)^{1/6}\\
&\qquad\qquad+Np^{-\varepsilon/4}\exp\((\log 2+o(1))\log R/\log \log R\).
\end{align*}
\end{thm}

It is evident from the proof of Theorem~\ref{thm:sqfree} that a slightly weaker condition on $R$ can also be imposed. However, this comes at the expense of a weaker bound in the second term of the estimate. Furthermore, the same argument applies for estimating $S_{\mu^{(k)}, \chi, P}$, where $\mu^{(k)}$ denotes the characteristic function of $k$-free numbers.

To understand the strengths of Theorem~\ref{thm:dirichlet} and Theorem~\ref{thm:sqfree} compared to Theorem~\ref{thm:divbdd}, at least for $1$-bounded multiplicative functions, consider the range of $N$ in \eqref{eqn:Nfor1bdd}. Certainly, within this range, both theorems save a power compared to the trivial bounds. In fact, over this range of $N$, we have a simplified estimate 
$$S_{\mu^2,\chi,P}(N)\ll Np^{-\varepsilon/4}\exp\((\log 2+o(1))\log R/\log \log R\).$$

\subsubsection{An estimation over smooth numbers}

We begin by recalling that a positive integer $n$ is called $y$-smooth if $P(n) \le y$, where $P(n)$ denotes the largest prime divisor of $n$. For $2 \le y \le N$, we denote by $\mathcal{S}(N,y)$ the set of $y$-smooth integers up to $N$. As usual, denote $\Psi(N,y) =\sharp\, \mathcal{S}(N,y)$. 

Let $\Psi_y$ denote the characteristic function of $y$-smooth numbers. Clearly, $\Psi_y$ is multiplicative, and we are interested in the sum
$$S_{\Psi_y,\chi,P}(N) = \sum_{n \in \mathcal{S}(N,y)} \chi(\psi_n(P)),$$
which admits the trivial upper bound $S_{\Psi_y,\chi,P}(N) \ll \Psi(N,y)$.

For various estimates of $\Psi(N,y)$ in various ranges of $y$, the reader may refer to \cite{Granv, Harp, HilTen86, HilTen}. In this article, we use the following estimate, which is stated in a form enough for our purposes. For some $\alpha(N,y)$, it is well known that
\begin{equation}\label{eq:psiNy}
\Psi(N,y) = N^{\alpha(N,y) + o(1)},
\end{equation}
which, for example, follows from \cite[Theorem 1]{HilTen86}, also mentioned in \cite[Section 1.3]{ShaShpWij}. Here $\alpha(N,y)$ satisfies 
\begin{equation}\label{eq: Psi N-alpha}
\alpha(N,y) = \(1+o(1)\)\frac{\log(1 + y/\log N)}{\log y}.
\end{equation}
See \cite[Theorem 2]{HilTen86} for a more precise asymptotic formula for $\alpha(N,y)$.

In particular, it is certainly useful to note the following facts
\begin{itemize}
    \item If $\log y / \log \log N \to \infty$, then $\alpha(N,y) = 1 + o(1)$.
    \item If $y = (\log N)^K$ for some fixed $K \ge 1$ (i.e., the case of \textit{very smooth} numbers), then
    $\alpha(N,y) = 1 - \frac{1}{K} + o(1)$, which also reflects the sparsity of \textit{very smooth}  numbers. 
\end{itemize}

\begin{thm}\label{thm:ysmooth}
   Let $R$ be as in \eqref{eq:Rweaker}, and $1\leq N\leq R$ be any integer. Then for any $y\ge (\log N)^{\frac{3+\sqrt{5}}{2}}$, we have
\begin{equation}\label{eq:charysmoothbound}
 S_{f_y,\chi,P}(N) \ll  y^{1/2} p^{\frac{1}{24(\alpha+2)}} R^{\frac{5}{12(\alpha+2)}} N^{-\gamma} \Psi(N,y) N^{o(1)}, 
\end{equation}
where  
$$\alpha=\alpha(N,y),\quad \mathrm{and}\quad\gamma=\frac{\alpha^2+\alpha-1}{2(\alpha+2)}.
$$
\end{thm}  

The condition on $y$ in Theorem~\ref{thm:ysmooth} is needed, because the estimate is meaningful when $\gamma > 0$, that is, when $\alpha(N, y) > \frac{\sqrt{5} - 1}{2}$. Furthermore, we obtain a power-saving in \eqref{eq:charysmoothbound}, provided that $N$ is in the range
\[N^{\alpha^2 + \alpha - 1} > p^{1/12} R^{5/6 + \varepsilon}, \quad \text{for some } \varepsilon > 0.\]

In fact, since $N\le R$ and $R\ll p$, we must then need that $\alpha^2+\alpha>23/12$, i.e., we need $\alpha>0.972$, in other words $y\ge (\log N)^{35.72}$.

In particular, when $\log y/\log \log N \to \infty$, we have $\alpha=1+o(1)$, and in that case we have a cleaner estimate of the form
$$S_{f_y,\chi,P}(N)\ll y^{1/2}p^{1/72}R^{5/36}N^{-1/6+o(1)}\Psi(N,y),$$
which is certainly non-trivial in the range $N\gg y^3p^{1/12}R^{5/6+3\varepsilon}$, for some $\varepsilon>0$, and also saves a power compared to the trivial bound. However, again since $N\le R$, we need $y\ll R^{1/18-\varepsilon}p^{-1/36}$.

\subsection{Notations} Throughout this article, as usual, the notations $U = O(V)$ and $U \ll V$ are equivalent to $|U|\le c V$ for some positive constant $c$, which may depend on $d$ (the order of $\chi$), and $q$ (the modulus of the Dirichlet character $\psi$).

We use $\tau(n)$ to denote the number of distinct prime and positive integer factors of an integer $n\ne 0$ and use $\varphi(n)$ to denote the Euler function.

We denote $\er(n) = \exp(2 \pi i n/R)$. Finally, we use $\sharp\, \cS$ to denote the cardinality of a finite set $\cS$, and throughout the rest of the article, we write $\sum\limits_{n\leq N}$ to denote the summation over positive integer $1 \leq n\leq  N$.

\section{Character sums with division polynomials}\label{sec:elliptic}
\subsection{Preliminary}
In this section, we establish the key results involving the division polynomials, which are needed for proving the results stated in Section~\ref{sec:results}.

As usual, for any $\Psi\in \F_p(E)$, we define
\[
\deg \Psi=\sum_{\substack{P\in E(\overline{\F_p})\\ \nu_P(\Psi)>0}} \nu_{P}(\Psi),
\]
where  $\nu_{P}(\Psi)$ is the multiplicity of $P$ as a zero of $\Psi$ (and thus this is a finite sum over 
all zeros of $\Psi$). 

Let us again recall $R$ from \eqref{eqn:R}. Note that,~\cite[Lemma~3.1]{ShSt} for $s=d$ implies that the sequences $\chi(\psi_n(P))$ is periodic with a period dividing $R$, as long as $\ord P\geq 3$. Recall the following estimate from~\cite[Lemma~2.3]{BhShp25}. 

\begin{lem}\label{lem:key}
Let $\chi$ be any non-principal multiplicative character of $\F_p^{*}$, and $R$ be as in~\eqref{eqn:R}. Assume that $\Psi$ is not a non-trivial power of a function in $\overline{\F_p(E)}$. Then, we have the following estimate uniformly over $a\in \Z$
$$\sum_{n\leq R} \chi(\Psi(nP))\er(an)\ll  \deg \Psi \cdot \sqrt{p}.$$
\end{lem}

We also need the following lemma to show that the functions of our requirements satisfy the condition of Lemma~\ref{lem:key}.
\begin{lem}\label{lem:non-power}
    Let $m_1,m_2,n_1,n_2$ be positive integers satisfying
    \begin{equation}\label{eqn:gcd}
    \gcd(m_1n_1,m_2n_2)=\gcd(m_1n_2,m_2n_1).
    \end{equation}
    Suppose further that
    \begin{equation}\label{eqn:distinct}
        m_1 \neq m_2 \quad \text{and} \quad n_1 \neq n_2.
    \end{equation}
    Then, the function 
    $$\Psi(Q)=\psi_{m_1n_1}\cdot \psi_{m_2n_2}\cdot \psi_{m_1n_2}^{-1}\cdot \psi_{m_2n_1}^{-1}(Q),$$
    is not a non-trivial power of any function in $\overline{\F_p(E)}$.
\end{lem}
\begin{proof}
For any integer $h$, note that $\psi_h$ has $h^2-1$ simple zeroes,
see~\cite[Exercise~III.3.7]{Silv2}.  

Clearly, then the product $\psi_{m_1n_1} \cdot \psi_{m_2n_2}$ has 
$$(m_1n_1)^2+(m_2n_2)^2-2\gcd(m_1n_1,m_2n_2)$$ many simple zeroes, and $\gcd(m_1n_1,m_2n_2)-1$ many zeroes of multiplicity $2$. Similarly, $\psi_{m_1n_2}^{-1} \cdot \psi_{m_2n_1}^{-1}$ has 
$$(m_1n_2)^2+(m_2n_1)^2-2\gcd(m_1n_2,m_2n_1)$$
many simple zeroes, and $\gcd(m_1n_2,m_2n_1)-1$ many of zeroes of multiplicity $2$. 

In particular, if $\Psi=G^{\nu}$ for some $\nu\geq 2$ and $G\in \overline{\F_p(E)}$, then each simple zero of $\psi_{m_1n_1} \cdot \psi_{m_2n_2}$ must also be a simple zero of $\psi_{m_1n_2}^{-1} \cdot \psi_{m_2n_1}^{-1}$, and the vice versa. Therefore, the total number of simple zeroes for both $\psi_{m_1n_1}\cdot \psi_{m_2n_2}$ and $\psi_{m_1n_2}^{-1} \cdot \psi_{m_2n_1}^{-1}$ coincides. We get a contradiction due to the imposed conditions at \eqref{eqn:gcd} and \eqref{eqn:distinct}.
\end{proof}

To establish the main result of this section, we need one more ingredient from~\cite[Lemma 3.2]{ShSt}, that demonstrates an \textit{almost multiplicative} nature of $\chi(\psi_n(P))$.

\begin{lem}\label{lem:mult}
    Let $\chi$ be a multiplicative character. Then for any integers $m,n$ we have
\[
\chi(\psi_{mn}(P))=\chi(\psi_m(nP))\chi(\psi_n(P))^{m^2}.
\]
\end{lem}

\subsection{Correlations of character values}

\begin{lem}\label{lem:corel}
Let $R$ be as in~\eqref{eq:Rweaker}, and $1\leq N\leq R$ be any integer. Then, for any distinct primes $\ell_1,\ell_2\leq R$, we have
    $$\sum_{n \leq N} \chi(\psi_{\ell_1n}(P))\overline{\chi(\psi_{\ell_2n}(P))}\ll \max\{\ell_1,\ell_2\}p^{\frac{1}{12}}R^{5/6}(\log R)(\log \log R)^{1/3}.$$
\end{lem}

\begin{proof}
Let $a$ be an integer. Denote
$$S_{\ell_1,\ell_2}(a)=\sum_{n \leq R} \chi(\psi_{\ell_1n}(P))\overline{\chi(\psi_{\ell_2n}(P))}\er\left(an\right).$$

For a parameter $L$ to be determined later, we consider the set of integers
$$\mathcal{R}_{d}=\left\{1\leq r \leq L:~r=1 \pmod d,\ \gcd(r,\ell_1\ell_2R)=1\right\}.$$ 

Clearly by the same argument as in the proof of \cite[Lemma 2.4]{BhShp25}, we have
$$
\sharp\,\cR_{d}  \ge  \frac{\varphi(\ell_1\ell_2 R)(L-1)}{d\ell_1\ell_2 R}-\tau(\ell_1\ell_2 R). 
$$ 

Since $\tau(\ell_1\ell_2R)\ll \tau(R)$ and $\ell_1,\ell_2\leq R$, by the standard estimates of the functions $\tau(n)$ and $\varphi(n)$ (see~\cite[Theorems~317 and~328]{HaWr}, for instance), we have 
\begin{equation}\label{eq:rdsize}
\sharp\,\cR_{d}  \gg \frac{L}{\log \log R},
\end{equation}
provided that 
 \begin{equation}\label{eq:Lchoose}
L \ge \exp\(0.7 \log p/\log \log p\). 
\end{equation}

Let us now set 
$$W =  \sum_{r\in \cR_{d}} \sum_{n\leq R} \chi(\psi_{ \ell_1 r n}(P))\overline{\chi(\psi_{\ell_2 r n}(P))} \er(ar n).$$
Since each $r\in \cR_{d}$ is co-prime to $R$, it is clear that
\begin{equation}\label{eqn:toW}
    S_{\ell_1,\ell_2}(a)=\frac{1}{\sharp\,\cR_d} W.
\end{equation}
Applying Cauchy-Schwarz combined with Lemma~\ref{lem:mult}, we get
\begin{equation}\label{eqn:ninside}
\begin{split}
    |W|^2 &\leq R \sum_{n\leq R} \left|\sum_{r\in \cR_{d}}\chi(\psi_{\ell_1 r  n}(P))\overline{\chi(\psi_{\ell_2 r  n}(P))} \er(ar n)\right|^2\\
     &= R \sum_{n\leq R} \left|\sum_{r\in \cR_{d}}\chi(\psi_{\ell_1 r }(nP))\overline{\chi(\psi_{\ell_2 r }(nP))}\chi(\psi_n(P))^{\ell_1^2-\ell_2^2} \er(ar n)\right|^2\\
    &\qquad\quad\quad= R \sum_{n\leq R} \left|\sum_{r\in \cR_{d}}\chi(\psi_{\ell_1 r }(nP))\overline{\chi(\psi_{\ell_2 r }(nP))} \er(ar n)\right|^2\\
    &\qquad\quad\quad\le R \sum_{r_1, r_2 \in \cR_{d}}\left|\sum_{n \leq R} \chi(\varPsi_{r_1,r_2}(nP)) \er(a(r_1 - r_2)n)\right|,
    \end{split}
\end{equation}
where $\varPsi_{r_1,r_2}\in \F_p(E)$ is given by the following
$$\varPsi_{r_1,r_2}(Q)= \psi_{\ell_1 r_1}(Q)\cdot \psi_{\ell_2 r_2}(Q)\cdot \psi_{\ell_2 r_1}^{-1}(Q)\cdot \psi_{\ell_1 r_2}^{-1}(Q).$$
Clearly, we have
\[
\deg \varPsi_{r_1,r_2} \ll  \(L\cdot \max\{\ell_1,\ell_2\}\)^2.
\]

Now for any $r_1\neq r_2\in  \cR_{d}$, we have
\begin{equation}\label{eq:gcdcond}
\gcd\(\ell_1 r_1,\ell_2 r_2\)=\gcd\(r_1,r_2\)=\gcd\(\ell_1r_2,\ell_2r_1\).
\end{equation}
Therefore, Lemma~\ref{lem:non-power} shows that $\varPsi_{r_1,r_2}$ is not a non-trivial power of any $G\in \overline{\F_p(E)}$. Applying Lemma~\ref{lem:key} to the inner sum in~\eqref{eqn:ninside} for each non-diagonal terms $r_1\neq r_2$, and trivially estimating the contributions from the diagonal terms $r_1=r_2$, we obtain
\begin{align*}
|W|^2& \leq R^2 \sharp\, \cR_d+R \sum_{r_1 \neq r_2 \in \cR_d}\left|\sum_{n\leq R} \chi(\Psi_{r_1,r_2}(nP)) \er(a(r_1 - r_2)n)\right|\\
&\ll R^2 \sharp\, \cR_d+R(\sharp\, \cR_d)^2 \cdot \deg \varPsi_{r_1,r_2}\cdot \sqrt{p}\\
&\ll R^2 \sharp\, \cR_{d}+R(\sharp\, \cR_d)^2 \(L\cdot \max\{\ell_1,\ell_2\}\)^2\cdot \sqrt{p}.
\end{align*}
In particular, \eqref{eq:rdsize} and \eqref{eqn:toW} imply that
\begin{equation*}
S_{\ell_1,\ell_2}(a)\ll R L^{-1/2}(\log \log R)^{1/2}+R^{1/2}L\max\{\ell_1,\ell_2\}p^{1/4}.
\end{equation*}

The proof concludes, choosing $L=R^{1/3}p^{-1/6}(\log \log R)^{1/3}$ (clearly $L$ satisfies \eqref{eq:Lchoose} because of the condition \eqref{eqn:Rcondition}), and by the completing technique as in~\cite[Section~12.2]{IwKow}.
\end{proof}

\begin{rem}\label{eq:quadchars}
Following the proof of Lemma~\ref{lem:corel}, it is easy to see that we do not need to assume that $\ell_1,\ell_2\leq R$. Moreover, if both $\ell_1,\ell_2$ are odd, and $\chi$ is quadratic, then just by Lemma~\ref{lem:mult} and Lemma~\ref{lem:key}, we immediately have 
    $$\sum_{n \leq N} \chi(\psi_{\ell_1n}(P))\overline{\chi(\psi_{\ell_2n}(P))}\ll \(\max\{\ell_1,\ell_2\}\)^2p^{1/2}(\log R).$$
    
Due to \eqref{eqn:Rcondition}, of course, the bound above is better than Lemma~\ref{lem:key}, as long as $\ell_1,\ell_2$ are fixed.
\end{rem}

\begin{rem}
If we take $m_1,m_2\leq R$ to be two distinct positive integers. Then, by the same argument in the proof of Lemma~\ref{lem:key}, we have the following weaker estimate
\begin{equation}\label{eq:chim1nm2n}
\sum_{n \leq N} \chi(\psi_{m_1n}(P))\overline{\chi(\psi_{m_2n}(P))}\ll \max\{m_1,m_2\}p^{\frac{1}{12}}R^{5/6}(\log R)^{4/3}.
\end{equation}

The reason for the weaker bound is that the identity \eqref{eq:gcdcond} does not necessarily hold, unless $r_1,r_2$ are primes. In that case, $\cR_d$ is essentially a set of primes, and we then have to argue as in the proof of \cite[Theorem 5.1]{ShSt}, where the argument does not allow us to replace the factor $(\log R)^{1/3}$ by $(\log\log R)^{1/3}$. 
\end{rem}

Following the argument in the remark above, we also have the following estimate for any $m$;
\begin{equation}\label{eq:withonlym}
\sum_{n \leq N} \chi(\psi_{m n}(P))\ll m p^{\frac{1}{12}}R^{5/6}(\log R)^{4/3}.
\end{equation}

In fact, one can get here $(\log R)(\log \log R)^{1/3}$ instead of $(\log R)^{4/3}$. This is because, $\psi_{mr_1}\cdot \psi_{mr_2}^{-1}$ is not power of any function in $\overline{\F_p}(E)$, for any integers $r_1\neq r_2$.

However, the dependence of $m$ in \eqref{eq:withonlym} would be too restrictive for us to prove Theorem~\ref{thm:sqfree}. To overcome that obstacle, we use Lemma~\ref{lem:mult} to get a sum associated with the point $mP\in E(\F_p)$ (see Lemma~\ref{lem:chimn} below). To handle the extra twisting factors $\chi(\psi_m(P))^{n^2}$ arising from Lemma~\ref{lem:mult}, we first need the following estimate, which essentially generalizes both \cite[Lemma 2.4]{BhShp25} and \cite[Theorem 5.1]{ShSt}.

\begin{lem}\label{lem:sumoverkmodd}
    Let $q$ and $1\leq k\leq q$ be any two fixed integers, and $R$ be as in \eqref{eq:Rweaker}. Then for any integer $N\leq R$, we have
    $$ \sum_{\substack{n\leq N\\ n\equiv k \pmod q}} \chi (\psi_n(P)) \ll p^{1/12}R^{5/6}(\log R)(\log \log R)^{1/3}.$$
\end{lem}

\begin{proof}
For a suitable parameter $L$, consider the set of integers
$$\mathcal{R}_{d,q}=\{r<L:r \equiv 1\pmod {dq}\},$$
where again, $d$ is the order of $\chi$. 

Since $q,d$ are both fixed, by the same argument as in the proof of \cite[Lemma 2.4]{BhShp25}, we have
$$
\sharp\,\cR_{d,q} \gg \frac{L}{\log \log R},\quad \mathrm{provided~that}\quad L \ge \exp\(0.7 \log p/\log \log p\).
$$ 
Then, for any $r\in \mathcal{R}_{d,q}$
$$\sum_{\substack{n\leq R\\n\equiv k \pmod q}} \chi(\psi_{n}(P))e_R(an)=\sum_{\substack{n\leq R\\n\equiv k \pmod q}} \chi(\psi_{nr}(P))e_R(anr).$$
In particular, the desired sum is bounded by $\frac{1}{\sharp\, \mathcal{R}_{d,q}}\cdot |W_k|$, where
$$W_k = \sum_{r \in \mathcal{R}_{d,q}}~\sum_{\substack{n\leq R\\n\equiv k \pmod q}} \chi(\psi_{nr}(P))e_R(anr).$$
By Cuachy-Schwarz and Lemma~\ref{lem:mult}, we have
\begin{align*}
|W_k|^2 &\le R \sum_{n\leq R}\left|\sum_{r \in \mathcal{R}_{d,q}}  \chi(\psi_{r n}(P))e_R(ar n)\right|^2\\
&=R\sum_{n\leq R}\left|\sum_{r \in \mathcal{R}_{d,q}} \chi(\psi_{r}(nP))\chi(\psi_n(P))e_R(a r n)\right|^2\\
&= R\sum_{n\leq R}\left|\sum_{r \in \mathcal{R}_{d,q}}  \chi(\psi_{r}(nP))e_R(a r n)\right|^2\\
&\le R \sum_{r_1,r_2 \in \mathcal{R}_{d,q}}\left|\sum_{n\leq R} \chi(\psi_{r_1}\cdot \psi_{r_2}^{-1}(nP))e_R(a(r_1-r_2) n)\right|.
\end{align*}

From this point, we argue similarly as exactly in the proof of \cite[Lemma 2.4]{BhShp25}, and conclude the proof.
    \end{proof}

Now, we are ready to weaken the dependence of $m$ in \eqref{eq:withonlym}.

\begin{lem}\label{lem:chimn}
Let $R$ and $\varepsilon$ be as in Theorem~\ref{thm:divbdd}, and $m\leq R$ be any positive integer satisfying $(m,R)\leq p^{\varepsilon/2}$. Then for any integer $1\leq N\leq R$, we have
    $$\sum_{n\leq N} \chi(\psi_{mn}(P))\ll p^{1/12}R^{5/6}(\log R)(\log \log R)^{1/3}.$$
\end{lem} 

\begin{proof}
By Lemma~\ref{lem:mult}, we can rewrite the sum as  
\[
\left|\sum_{n\leq N} \chi(\psi_n(mP))\chi(\psi_m(P))^{n^2}\right| \ll \sum_{k \leq d} \left| \sum_{\substack{n \leq N \\ n \equiv k \pmod{d}}} \chi(\psi_n(mP)) \right|.
\]

Now, applying Lemma~\ref{lem:sumoverkmodd} to each of the inner sums with $q = d$, and noting that
\[
\ord(mP) = \frac{\ord(P)}{(m, \ord(P))} \geq p^{1/2 + \varepsilon/2},
\]
the proof follows.
\end{proof}
\section{Proof of Theorem~\ref{thm:divbdd}}
Let us recall the sum of our interest
\begin{equation*}
    S_{f,\chi,P}(N) = \sum_{n\leq N} f(n)\chi(\psi_n(P)).
\end{equation*}

To estimate the sum, we closely follow the approach outlined in \cite{KorShp}. Let $2 \leq x < y \leq N$ be parameters to be chosen appropriately, and consider the interval $I = (x, y]$. We then split the sum according to whether the indices have no prime factors in $I$, or have at least one prime factor in $I$, and estimate each contribution separately.

To make the splitting more systematic, denote $\cA_r(N, I)$ the set of integers $1\leq n \leq N$ with exactly $r$ prime factors (counted with multiplicities) in $I$. Note that $\cA_r(N,I)=\emptyset$, for any $r>\log N$. Therefore, we can write
\begin{equation}\label{eqn:sumoverr}
S_{f, \chi,P}(N) = \sum_{r\leq \log N} U_{f,r}(N,I),
\end{equation}
where we denote
$$U_{f,r}(N,I)=\sum_{n\in \cA_r(N, I)} f(n)\chi(\psi_n(P)).$$

For the rest of this section, we abuse the notation by writing $U_r$ in place of $U_{f,r}(N,I)$.

\subsection{Treatment for $U_0$}
The fundamental lemma of combinatorial sieve \cite[Theorem 4.4]{Tenenbaum} shows that the set $\cA_0(N,I)$ is small when $x$ and $y$ are widely separated parameters. However, we need an estimate for the sum of $\tau_{\nu}$-bounded functions over this small set. We simply use the following estimate, readily available from \cite[(8-2) and (8-3)]{KorShp}.
\begin{equation}\label{eq:u0}
U_{0}\ll N \exp\left(-\frac{\log N}{4\log x}\right) (\log N)^{\frac{\nu^2-1}{2}}+\frac{N(\log N)^{\nu-1}}{(\log y)^{\nu}}(\log x)^{\nu}.
\end{equation}
A similar analysis could also be found in \cite[Lemma 2.2]{Sun}.

\subsection{Treatment for $U_r$ for $r>0$}

We can split the sum $U_r$ as 
$$U_r = U_{r,1} + U_{r,2},$$
where $U_{r,1}$ denotes the contribution from the integers $n \in \cA_r(N,I)$ whose prime divisors from the interval $I$ all appear with multiplicity one. The remaining contribution, coming from integers where at least one prime divisor from $I$ appears with multiplicity at least two, is denoted by $U_{r,2}$. 

Note that the sums $U_{r,2}$ are much easier to handle, as crudely we have
\begin{equation}\label{eqn:u2tau}
\begin{split}
\left|\sum_{r\leq \log N} U_{r,2}\right| &\leq \sum_{\substack{x< \ell\leq y\\ \ell~\mathrm{prime}}}~\sum_{1\leq s \leq \frac{N}{\ell^{2}}} \left|f(\ell^2 s)\right|\leq  \sum_{\substack{x< \ell\leq y\\ \ell~\mathrm{prime}}} 
\tau_{\nu}(\ell^{2}) 
\sum_{1\leq s \leq \frac{N}{\ell^{2}}}
\tau_{\nu}(s)\\
&\leq N (\log N)^{\nu-1}\sum_{\substack{x< \ell\\ \ell~\mathrm{prime}}} 
\frac{\tau_{\nu}(\ell^{2})}{\ell^2}\ll \frac{N(\log N)^{\nu-1}}{x \log x}.
\end{split}
\end{equation}
\subsubsection{Estimating {$U_{r,1}$}}
Denote $\cA'_r(N,I)$ be the set of integers that appear as an index in the sum $U_{r,1}$. Note that any $n\in \cA'_r(N,I)$ has exactly $r$ representations
of the form $n = \ell m$, where the prime $\ell \in I$, and integer $m\in \cA'_{r-1}(N/\ell,I)$, with $\gcd(m,\ell)=1$. Due to the multiplicativity of $f$, we can write
$$
U_{r,1}=r^{-1}\sum_{\ell \in I}~f(\ell)~\sum_{\substack{m\in \cA'_{r-1}(N/\ell,I)\\\g(m, \ell) = 1}} f(m)\chi(\psi_{\ell m}(P)).$$

Now, dividing $I=(x,y]$ into $K=O(\log y)$ many dyadic intervals $I_k=(x_k,y_k]$ for $k=0,1,\cdots, K$, as in \cite[(7-7)]{KorShp}, we can write 
\begin{equation}\label{eq:usplit}
U_{r,1}=\sum_{k=0}^{K} V_{k,r},
\end{equation}
where 
\begin{align*}
V_{k,r}&=r^{-1}\sum_{\ell\in I_k}~f(\ell)~\sum_{\substack{m\in \cA'_{r-1}(N/\ell, I)\\\g(m, \ell) = 1}} f(m)\chi(\psi_{\ell m}(P))\\
&=r^{-1}\sum_{\substack{m\in \cA_{r-1}(N/x_k, I)}}~f(m)~\sum_{\substack{\ell \in I_k\cap (1,N/m] \\\g(m, \ell) = 1}} f(\ell) \chi(\psi_{\ell m}(P)).
\end{align*}

Note that $|f(m)|\le \tau_{\nu}(m)$, and for each $m\in  \cA'_{r-1}(N/x_k,I)$ there exists at most $r$ many $\ell \in I_k$ for which $\gcd(m,\ell)\neq 1$. Therefore, we can write
\begin{equation}\label{eq:vbound}
V_{k,r}\ll r^{-1}W_{k,r}+N(\log N)^{\nu-1}x_k^{-1},
\end{equation}
where
\begin{equation*}
W_{k,r}=\sum_{m\in \cA_{r-1}\left(N/x_k,I\right)}\tau_{\nu}(m)\left|\sum_{\ell \in I_k\cap (1,N/m]}f(\ell)\chi(\psi_{\ell m}(P))\right|.
\end{equation*}

Crudely, summing over $m\leq N/x_k$ and applying Cauchy–Schwarz, we have
\begin{align*}
&|W_{k,r}|^2\leq Nx_k^{-1}(\log N)^{\nu^2-1} \sum_{m \leq N/x_k} 
    \left| \sum_{\ell \in I_k \cap (1,N/m]} 
    f(\ell)\, \chi(\psi_{\ell m}(P)) \right|^2 \\
&\leq Nx_k^{-1}(\log N)^{\nu^2-1} \sum_{\ell_1,\,\ell_2 \in I_k} 
    \left| \sum_{m \leq N/{\max\{x_k,\, \ell_1,\, \ell_2\}}} 
    \chi(\psi_{\ell_1 m}(P))\, \overline{\chi(\psi_{\ell_2 m}(P))} \right|, 
    \end{align*}
where we are using above that $|f(\ell_1)f(\ell_2)|= O(1)$, and the estimate from \eqref{eq:sumtau} for $C=2$.

At this point, our argument deviates from the proof of \cite[Theorem 2.1]{KorShp}. Instead of using \cite[Corollary 4.2]{KorShp}, we appeal to Lemma~\ref{lem:key} to estimate the inner sums above for $\ell_1\neq \ell_2$. Moreover, we estimate the contributions from $\ell_1=\ell_2$ trivially, and get
\begin{equation*}
\begin{split}
\left|W_{k,r}\right|^2&\ll \left(N x_k^{-1}\left(y_kNx_k^{-1}+y^3_k p^{\frac{1}{12}}R^{5/6+o(1)}\right)\right)(\log N)^{\nu^2-1}\\
&\ll \left(N^2x_k^{-1}+Nx_k^2 p^{\frac{1}{12}}R^{5/6+o(1)}\right)(\log N)^{\nu^2-1}.
\end{split}
\end{equation*}
Consequently, we deduce the following estimate from \eqref{eq:usplit} and \eqref{eq:vbound}
\begin{equation}\label{eqn:u1tau} 
\begin{split}
U_{r,1}&\ll r^{-1}\sum_{k=0}^{K}\(W_{k,r}+rN(\log N)^{\nu-1}x_k^{-1}\)\\
&\ll r^{-1} \left( Nx^{-1/2} + y N^{1/2} p^{\frac{1}{24}}R^{5/12+o(1)}\right)(\log N)^{\frac{\nu^2-1}{2}}\\
&\hspace{5cm}+N(\log N)^{\nu-1}x^{-1}.
\end{split}
\end{equation}

\subsection{Concluding the proof}\label{sec:concludingThm1}
From \eqref{eqn:sumoverr}, we have
\begin{align*}
S_{f, \chi,P}(N)&\ll \sum_{\substack{r\leq \log N \\r \neq 0}} |U_{r,1}|+\left|\sum_{\substack{r\leq \log N \\r\neq 0}} U_{r,2}\right|+|U_{0}|.
\end{align*}
Let us choose the parameters
$$x=(\log R)^{\nu^2+3},\quad \mathrm{and} \quad y=R^{\varepsilon/4}.$$
From \eqref{eqn:u1tau}, we have 
\begin{align*}
\sum_{\substack{r\leq \log N \\r \neq 0}}|U_{r,1}|&\ll \frac{N} {\log R}+y N^{1/2} p^{\frac{1}{24}}R^{5/12+o(1)}\ll \frac{N}{\log R},
\end{align*}
where the last inequality is due to the condition \eqref{eqn:Nfor1bdd}. 

From \eqref{eqn:u2tau}, we have
$$\left|\sum_{\substack{r\leq \log N\\r\neq 0}} U_{r,2}\right|\ll \frac{N(\log N)^{\nu-1}}{x \log x}\ll \frac{N}{\log R}.$$
The proof is now complete, as \eqref{eq:u0} implies
$U_0\ll \varepsilon^{-\nu}N\frac{(\log \log R)^{\nu}}{\log R}.$
\qed

\section{A power saving with certain twists}
 \subsection{Proof of Theorem~\ref{thm:dirichlet}}
To prove this, we do not quite follow the techniques of the previous section; rather, we simply apply Lemma~\ref{lem:sumoverkmodd}. Suppose that $\psi$ is a Dirichlet character modulo $q$. Then, we can write
\begin{align*}
S_{\psi,\chi,P}(N)=\sum_{k\leq q}\psi(k)\sum_{\substack{n\leq N\\ n\equiv k \pmod q}} \chi(\psi_{n}(P)).
\end{align*}
The proof follows, immediately applying Lemma~\ref{lem:sumoverkmodd} to each of the $q$ many inner sums.
\qed

\subsection{Proof of Theorem~\ref{thm:sqfree}}
Let us first write
$$S_{\mu^2,\chi,P}(N) = \sum_{d\leq N^{1/2}} \mu(d)\sum_{n\leq N/d^2}\chi(\psi_{d^2n}(P)).$$
Then, we can split the sum as 
\begin{equation}\label{eq:split}
S_{\mu^2,\chi,P}(N)=T_{1}+T_{2},
\end{equation}
where $T_{1}$ is the contribution from the $d$ for which $\gcd(d,R)\leq p^{\varepsilon/4}$, and $T_2$ be the sum over the remaining $n\leq N$. 

First, we split $T_1$ into $\log N$ many type I sums of the form
$$S(D,N)=\sum_{\substack{D\leq d<2D\\ \gcd(d,R)\leq p^{\varepsilon/4}}} \mu(d)\sum_{n\leq N/d^2}\chi(\psi_{d^2n}(P)).$$

As for any such $d$ above we have $\gcd(d^2,R)\leq p^{\varepsilon/2}$, Lemma~\ref{lem:chimn} implies that
\begin{equation}\label{eq:sdn1}
    S(D,N)\ll Dp^{1/12}R^{5/6}(\log R)(\log \log R)^{1/3}.
\end{equation}
Estimating rather trivially, we also have
\begin{equation*}
S(D,N)\ll N/D.
\end{equation*}

In particular, choosing $D_0$ such that 
$$D_0^2p^{1/12}R^{5/6}(\log R)(\log \log R)^{1/3}=N,$$ 
and applying \eqref{eq:sdn1} for $D\le D_0$, and \eqref{eq:sdn1} for $D\ge D_0$, we finally derive
$$T_1\ll N^{1/2}p^{1/24}R^{5/12}(\log R)^{3/2}(\log \log R)^{1/6}.$$

On the other hand, the number of $n\leq N$ that are divisible by at least one divisor of $R$ greater than $p^{\varepsilon/4}$, is at most $Np^{-\varepsilon/4}\tau(R)$. Here, $\tau(R)=\tau_{1}(R)$ denotes the number of divisors of $R$. In particular, the trivial estimation gives 
$$T_2\ll Np^{-\varepsilon/4}\exp\((\log 2+o(1))\log R/\log \log R\),$$
where we use the well-known bound on $\tau(R)$; see~\cite[Theorem 317]{HaWr}. The proof concludes from \eqref{eq:split}.
\qed

\section{Proof of Theorem~\ref{thm:ysmooth}}
\subsection{Preliminaries}
\subsubsection{Smooth numbers in a smaller range} We first need the following estimates.

\begin{lem}\label{lem:psismallN}
    Let $\alpha=\alpha(N,y)$ be as in \eqref{eq: Psi N-alpha}. For any $1\le L\le N$, we have
    \begin{equation}\label{eq:N/L}
        \Psi(N/L,y) \ll \frac{1}{L^{\alpha}} \Psi(N,y),\quad \mathrm{and}
    \end{equation}
    \begin{equation}\label{eq:L}
         \Psi(L, y) \ll L^{\alpha}N^{o(1)}.
    \end{equation}
\end{lem}
\begin{proof}
    For \eqref{eq:N/L}, see \cite[Section~2]{Harp}. To prove \eqref{eq:L}, simply note that
\begin{equation}\label{eq:psiL}
 \Psi(L, y) \ll \left(\frac{L}{N}\right)^{\alpha}\Psi(N,y)= L^{\alpha}N^{o(1)},
 \end{equation}
where, of course, we are using \eqref{eq:psiNy} for the last equality.
\end{proof}
\subsubsection{A decomposition of large smooth numbers}
Let $L_0 \leq N$ be a parameter to be chosen later. Note that any $n\in \cS(N,y)$ with $n> L_0$ can be written as
$$n = \ell m,\quad \mathrm{with}\quad L_0<\ell \le P(\ell)L_0 \quad \mathrm{and}\quad p(m) \geq P(\ell), $$
where $P(\ell)$ denotes the largest prime factor of $\ell$ and $p(m)$ denotes the smallest prime factor of $m$. Indeed, one can obtain such a decomposition by arranging all the prime factors of $n$ as $p_1\le p_2 \cdots \le p_k$, and keep on collecting the numbers $p_1,p_1p_2,\cdots$, until we reach a number $\ell$ that just crosses $L_0$. See also \cite{MerShp,ShaShpWij,Vau89}.

Then, we can write 
$$S_{f_y,\chi,P}(P)=\sum_{\substack{L_0 < \ell \leq P(\ell) L_0 \\ \ell \in \cS(N, y)}} \, \sum_{\substack{m \in \cS(N/\ell, y) \\ p(m) \geq P(\ell)}} \chi(\psi_{\ell m}(P)) + O(L_0).$$

\subsection{Concluding the proof}
Note that the range of $\ell$ is given by $(L_0,yL_0]$. After the dyadic partition of this range, we see that there is some $L \in (L_0, yL_0]$ such that 
\begin{equation}\label{eq:StoU}
S_{f_y,\chi,P}(P)  \ll U \log N + L_0, 
\end{equation}
where 
$$ U= \sum_{\substack{L < \ell \leq \min\left\{P(\ell)L_0, 2L\right\} \\ \ell \in \cS(N, y)}} \,  \sum_{\substack{m \in S(N/\ell, y) \\ p(m) \geq P(\ell)}} \chi(\psi_{\ell m}(P)).$$

We now argue exactly as in the proof of \cite[Theorem 1.3]{ShaShpWij}, and obtain
\begin{equation*}
\begin{split}
U^2 \ll  & (\log N)^2 \Psi(L,y) \\ 
&\quad\quad\times \sum_{q \leq y} \, \sum_{m_1,m_2 \in \cS(N/L,y)} \left| \sum_{L/q < \ell \le 2L/q} \chi(\psi_{q\ell m_1}(P)) \overline{\chi(\psi_{q\ell m_1}(P)} \right|. 
\end{split} 
\end{equation*}

For the rest of the proof, we set $\alpha= \alpha(N,y)$. Following the same argument as in \cite{ShaShpWij}, the contribution $Y_1$ from  the diagonal terms with $m_1=m_2$ is 
\begin{align*}
Y_1&\ll (\log N)^2 \Psi(L,y) \sum_{q \leq y} \Psi(N/L, y) L/q\\
&\ll L N^{-\alpha} \Psi(N,y)^2 N^{o(1)},
\end{align*}
where the last inequality follows from \eqref{eq:N/L} and the first inequality in \eqref{eq:psiL}.

To estimate the contribution $Y_2$ from the non-diagonal terms $m_1\neq m_2$, we apply \eqref{eq:chim1nm2n} instead of \cite[Lemma 2.2]{ShaShpWij}, and obtain
\begin{align*}
Y_2 & \ll (\log N)^2 \Psi(L,y)\cdot  y \Psi(N/L, y)^2\cdot N/L\cdot p^{\frac{1}{12}}R^{5/6+o(1)} \\
& \ll y p^{1/12}R^{5/6} L^{-(\alpha+1)} \Psi(N,y)^2 N^{1+o(1)}, 
\end{align*}
where the last inequality follows by combining both parts of Lemma~\ref{lem:psismallN}.

In particular, recalling that $L_0<L\le yL_0$, we have
\begin{align*}U^2& \ll Y_1 + Y_2   \le \(LN^{-\alpha} + yp^{1/12}R^{5/6}L_0^{-(\alpha+1)}N\) \Psi(N,y)^2 N^{o(1)} \\
& \le  y\(L_0N^{-\alpha} + p^{1/12}R^{5/6}L_0^{-(\alpha+1)}N\) \Psi(N,y)^2 N^{o(1)}. 
\end{align*}
Choosing $L_0 = p^{\frac{1}{12(\alpha+2)}} R^{\frac{5}{6(\alpha+2)}}N^{\frac{\alpha+1}{\alpha+2}}$, we derive
$$U \le y^{1/2} p^{\frac{1}{24(\alpha+2)}} R^{\frac{5}{12(\alpha+2)}}N^{\frac{1-(\alpha^2+\alpha)}{2(\alpha+2)}} \Psi(N,y) N^{o(1)}.$$

Hence, \eqref{eq:StoU} gives
\begin{align*}
S_{f_{y},\chi,P}(N)&\ll y^{1/2} p^{\frac{1}{24(\alpha+2)}} R^{\frac{5}{12(\alpha+2)}}N^{\frac{1-(\alpha^2+\alpha)}{2(\alpha+2)}} \Psi(N,y) N^{o(1)}\\
&\qquad\qquad\qquad\qquad+p^{\frac{1}{12(\alpha+2)}} R^{\frac{5}{6(\alpha+2)}}N^{\frac{\alpha+1}{\alpha+2}}.
\end{align*}

Setting $\gamma=\frac{\alpha^2+\alpha-1}{2(\alpha+2)}$, we have $\alpha-2\gamma=\frac{\alpha+1}{\alpha+2}$, and hence we finally deduce
\begin{align*}
S_{f_{y},\chi,P}(N) &\ll y^{1/2} p^{\frac{1}{24(\alpha+2)}} R^{\frac{5}{12(\alpha+2)}} N^{-\gamma} \Psi(N,y) N^{o(1)} \\
&\qquad\qquad\qquad\qquad + p^{\frac{1}{12(\alpha+2)}} R^{\frac{5}{6(\alpha+2)}} N^{\alpha - 2\gamma}.
\end{align*}

If $p^{\frac{1}{24(\alpha+2)}} R^{\frac{5}{12(\alpha+2)}}> N^{\gamma}$, then the stated bound at \eqref{eq:charysmoothbound} is trivial. So, let us assume that $p^{\frac{1}{24(\alpha+2)}} R^{\frac{5}{12(\alpha+2)}}\le N^{\gamma}$. In this case, the second term above is dominated by the first term because $\Psi(N,y)N^{o(1)}=N^{\alpha}$, and the proof concludes.
\qed

\begin{rem}\rm
Following the arguments in Section~\ref{sec:concludingThm1}, the main bottleneck in improving Theorem~\ref{thm:divbdd} lies in the estimation of $U_0$ in \eqref{eq:u0}. Note that the integers in the set $\cA_0(N,I)$ can be written as products of the form
\[
(\text{$x$-smooth}) \times (\text{$y$-rough}).
\]
This structure suggests that one might attempt to adapt the approach used in the proof of Theorem~\ref{thm:ysmooth}. However, as one may anticipate, complications arise for the integers whose $x$-smooth parts are small.
\end{rem}

\begin{rem}\rm
 Consider the classical function $r_0(n)$, the characteristic function of integers that are sums of two squares. It is well-known (see \cite[Equation~(4.90)]{Tenenbaum}, for instance) that
\[
\sum_{n \leq N} r_0(n) \ll \frac{N}{(\log N)^{1/2}},
\]

Following the argument used in the proof of Theorem~\ref{thm:divbdd}, one can easily deduce that
\[
S_{r_0,\chi,P} \ll \varepsilon^{-1} \frac{N}{(\log N)^{1/2}} \cdot \frac{\log \log R}{\log R}.
\]

Moreover, the same savings compared to the trivial estimate likely holds for many other multiplicative functions, \textit{supported sparsely} on the integers and taking values in the intervals $[0,1]$.
\end{rem}

\section*{Acknowledgement}
The author is grateful to Igor E. Shparlinski for suggesting this project, and to both Alina Ostafe and Igor E. Shparlinski for their valuable support and suggestions throughout the various stages of this work. The author also thanks Jitendra Bajpai and Simon L. Rydin Myerson for many helpful comments. This research was partially supported by the Australian Research Council Grant DP230100530.

\end{document}